\documentclass{amsart}
\usepackage{amssymb}
\usepackage{graphicx}
\newtheorem{theorem}{Theorem}[section]

\theoremstyle{definition}

\theoremstyle{remark}
\newtheorem{remark}[theorem]{Remark}

\numberwithin{equation}{section}

\begin{document}
\title[Global stability]{Global stability\\for a coupled physics inverse problem}
\author{Giovanni Alessandrini}
\address{Dipartimento di Matematica e Geoscienze, Universit\`a degli Studi di Trieste, Italy}

\email{alessang@units.it}
 
\thanks{This work is supported by FRA2012 `Problemi Inversi', Universit\`a degli Studi di Trieste}

\subjclass[2000]{} \keywords{}

\begin{abstract} We prove a global H\"older stability estimate for a hybrid inverse problem combining microwave imaging and ultrasound. The principal features of this result are that we assume to have access to measurements associated to a single, arbitrary and possibly sign changing solution of a Schr\"odinger equation, and that zero is allowed to be an eigenvalue of the equation.
\end{abstract}

\maketitle
\section{Introduction}
In this note we consider an inverse problem with internal measurements. Given the Schr\"odinger equation
\begin{equation}\label{sch}
\Delta u + q u = 0 \, \text{ in } \Omega \ , 
\end{equation}
in a bounded Lipschitz domain $\Omega$, the inverse problem consists of finding the coefficient $q=q(x)\ge const.>0$ given the interior measurement $q u^2$ and the boundary data $u|_{\partial \Omega}$ for one nontrivial solution $u$.
This is a simplified version of an hybrid inverse problem introduced by
	Ammari, Capdeboscq, De Gournay, Rozanova-Pierrat and Triki \cite{ammarieco}  in which the internal electromagnetic parameters of a body are to be detected by illuminating it by microwaves and, simultaneously,  by focusing ultrasonic waves  on a small portion of it. In mathematical terms, they deal with the equation
\begin{equation}\label{sch2}
div ({a}\nabla u)+ {k^2}q u = 0 \, \text{ in } \Omega  \ ,
\end{equation}
where $\Omega$  represents the body, $a^{-1}>0$ models the magnetic permeability, $q>0$ the electric permittivity and $u$ the electric field.
The goal in \cite{ammarieco} is  to find  ${a}, q$ given the local energies $q u^2$ , $a|\nabla u|^2$ when such measurements are available for \emph{several} solutions $u$ and wavenumbers $k$. 

For the simplified equation \eqref{sch}, Triki \cite {triki} obtained a uniqueness result for the determination of $q$ given the local energy $q u^2$ for one solution $u$ corresponding to prescribed Dirichlet data,  in \cite {triki} also one result of local stability is found.

In this note we shall investigate global stability in the same setting.

In contrast to the general  phenomenology of inverse boundary value problems and of inverse scattering problems, which typically show severe ill-posedness ( Mandache \cite{mandacheIP}, Di Cristo and Rondi  \cite{dicrondi}), the emerging methodologies  of imaging based on the coupling of different physical modalities (see Bal  \cite{balinside} for an overview), enable to acquire interior measurements which lead to the expectation of a much better behavior in terms of stability. Such an expectation has indeed been confirmed in many cases, let us cite Bal and Uhlmann \cite{baluhl},  Kuchment and Steinhauer \cite{kush}, Montalto and Stefanov \cite{montstef}, just to mention, with no ambition of completeness, a few samples of an impressively growing literature. However, a recurrent feature of the available results of stability is that suitable nondegeneracy conditions on the solutions of the underlying equations are needed. Depending on the problem treated, and on the number of coefficients that  are to be determined,  it  is required that the solution $u$\, or its gradient, does not vanish. In other cases it is required that the Jacobian matrix associated to an array of solutions is nonsingular, and even bigger matrices formed by an array of solutions and their derivatives may come into play. Such nondegeneracy conditions may be difficult to achieve in general. Typically, it is possible to prove the existence of nondegenerate sets of solutions, but,  in presence of unknown coefficients, it may be unclear how one can drive the system from the exterior so that the solutions satisfy the desired nondegeneracy conditions. It should be mentioned that, if one is free to tune up one parameter in the governing equation, like the wavenumber $k$ in \eqref{sch2}, then nondegeneracy can be guaranteed for at least some choice of $k$, this is a remarkable recent result by Alberti \cite{alberti}. 

Here,  instead, we intend to examine the prototypical case \eqref{sch} when measurements can be taken at one fixed wavenumber only, conventionally set as $k=1$. For such an equation
the associated, Dirichlet or Neumann, direct boundary value problems might be not well posed if $0$ is an eigenvalue, or well posed with very large costants, if $0$ is close to an eigenvalue. Hence it may be troublesome to control the interior behavior of a solution by appropriate choices of boundary data. And, generically,  solutions to \eqref{sch} may vanish inside, and may have critical points.

Let us show, by a simple example in dimension $n=1$, what kind of pathologies one might encounter.

Let us fix $0<r<R$ and, for every $m=1,2,\ldots$, let us set
\begin{eqnarray*}
q_m(x)  = 
\left\{
\begin{array}{rl}
A_m & \ \ \ \  \mbox{if}\ |x|<r \ ,\\
 1 &\ \ \ \ \mbox{if}\ r\le |x| \le R\ ,
\end{array}
\right.
\end{eqnarray*}
where
\begin{equation*}
A_m= \left(\frac{\pi}{2}+2m\pi\right)^2r^{-2} \ . 
\end{equation*}
A solution to $u_{xx}+q_m u = 0$  in $(-R,R)$ is
\begin{eqnarray*}
u_m(x)  = 
\left\{
\begin{array}{rl}
\frac{1}{\sqrt{A_m}}\cos (\sqrt{A_m}x) & \ \ \ \  \mbox{if}\ |x|<r \ ,\\
 -\sin (|x|-r) &\ \ \ \ \mbox{if}\ r\le |x| \le R\ .
\end{array}
\right.
\end{eqnarray*}
Note that in the interval $(-r,r)$ $u_m$ becomes very small and highly oscillating as $m$ increases. We have
\begin{equation*}
\|q_{2m}u_{2m}^2- q_{m}u_{m}^2\|_{\infty} \le 2 \text{ for every } m=1,2,\ldots \ ,
\end{equation*}
 whereas, for any $p \ , 1 \le p \le \infty$
\begin{equation*}
\|q_{2m}- q_{m}\|_{p} \rightarrow \infty \text{ as } m\rightarrow \infty \ .
\end{equation*}
Thus, in other words, the error on the measurement $qu^2$ does not dominate the error on $q$.  

Here we shall present a global stability result of conditional type,  which makes use of measurements for only one arbitrary, nontrivial, but possibly sign changing, solution, and when no spectral condition of the underlying equation is assumed, that is, we admit that $0$ might be an eigenvalue for equation \eqref{sch}. 

Let us remind that for a different hybrid problem, coupling elastography and magnetic resonance, Honda, McLaughlin and Nakamura \cite{nakajoice} also obtained a global stability of H\"older type, when the underlying solution may vanish somewhere. There are,  however, substantial differences in the governing equations which impose much different a priori assumptions and methods.

Let us now describe the a priori assumptions that we shall use. On the unknown coefficient $q\in L^{\infty}(\Omega)$ we require that for a given $K\ge 1$
we have
\begin{equation}\label{K}
K^{-1} \le q \le K \ ,  \text{ a.e. in } \Omega \ .
\end{equation}
Also, we consider one weak solution  $u\in W^{1,2}(\Omega)\cap C(\overline{\Omega})$ to \eqref{sch} and we prescribe for a given $E>0$ the following global energy bound
\begin{equation}\label{E}
\int_\Omega\left(u^2+ |\nabla u|^2\right) \le E^2\ .
\end{equation}
Regarding the interior measurement $qu^2$ we require the following nondegeneracy in average, that is we are given $H>0$ such that
\begin{equation}\label{H}
\int_\Omega qu^2 \ge H^2\ .
\end{equation}
Further we shall also need to specify in a quantitative form the Lipschitz regularity of the domain $\Omega$. For this purpose some notation and definitions are needed.

Given $x\in \mathbb R^n$, we shall denote $x=(x',x_n)$, where $x'=(x_1,\ldots,x_{n-1})\in\mathbb R^{n-1}$, $x_n\in\mathbb R$.
Given $x\in \mathbb R^n$, $r>0$, we shall use the following notation for balls and cylinders
\begin{equation*}
   B_r(x)=\{y\in \mathbb R^n\ |\ |y-x|<r\}, \quad  B_r=B_r(0)\ ,
\end{equation*}
\begin{equation*}
   B'_r(x')=\{y'\in \mathbb R^{n-1}\ |\ |y'-x'|<r\}, \quad  B'_r=B'_r(0)\ ,
\end{equation*}
\begin{equation*}
   \Gamma_{a,b}(x)=\{y=(y',y_n)\in \mathbb R^n\ |\ |y'-x'|<a, |y_n-x_n|<b\}, \quad \Gamma_{a,b}=\Gamma_{a,b}(0)\ .
\end{equation*}

We shall say that 
$\Omega$ is of \emph{Lipschitz class} with
constants $\rho$, $M>0$, if, for any $P \in \partial\Omega$, there exists
a rigid transformation of coordinates under which $P=0$
and
\begin{equation*}
  \Omega \cap \Gamma_{\frac{\rho}{M},\rho}(P)=\{x=(x',x_n) \in \Gamma_{\frac{\rho}{M},\rho}\quad | \quad
x_{n}>Z(x')
  \}\ ,
\end{equation*}
where $Z:B'_{\frac{\rho}{M}}\to\mathbb R$ is a Lipschitz function satisfying
\begin{equation*}
Z(0)=0,
\end{equation*}
\begin{equation*}
\|Z\|_{{L}^{\infty}(B'_{\frac{\rho}{M}})}+
  \rho\|\nabla Z\|_{{L}^{\infty}(B'_{\frac{\rho}{M}})} \leq M\rho \ .
\end{equation*}

We shall also use the following notation.

For every $d>0$ we denote
\begin{equation*}
\Omega_d =\left\{x\in \Omega\  | \ dist(x, \partial \Omega) > d \right\} \ .
\end{equation*}

By $|\Omega|$ we shall denote the measure of $\Omega$.

We can now state our main result.

\begin{theorem}\label{main} Let $\Omega$ be of {Lipschitz class} with
constants $\rho$, $M>0$. Let $q_1, q_2\in L^{\infty}(\Omega)$ satisfy \eqref{K}, let $u_1,u_2\in W^{1,2}(\Omega)\cap C(\overline{\Omega})$ be solutions to \eqref{sch} when $q=q_1, q_2$ respectively. Assume that $u_1,u_2$ satisfy the a priori assumptions \eqref{E}, \eqref{H} and suppose that, 
for a given $\varepsilon >0$ ,
\begin{equation}\label{interror}
 \|q_1u_1^2-q_2u_2^2\|_{L^{\infty}(\Omega)} \le \varepsilon \ ,
\end{equation}
 and also
\begin{equation} \label{bdryerror}
\||u_1|-|u_2|\|_{L^{\infty}(\partial\Omega)} \le \sqrt{K\varepsilon} \ .
\end{equation}
 Then, for every $d>0$, there exists $\eta\in (0,1)$ and $C>0$, only depending on $d, K, E, H$ and on $\rho, M, |\Omega|$ such that
\begin{equation}\label{intstab}
 \|q_1-q_2\|_{L^{1}(\Omega_d)} \le C\left(\varepsilon^{1/2}+\varepsilon \right)^{\eta} \ ,
\end{equation}
\end{theorem}
\begin{remark}

We observe that \eqref{intstab} provides a global stability of H\"older type. The H\"older exponent  is expected to depend on the a priori data and it might get smaller and smaller as the a priori bounds deteriorate. 

Moreover, let us remark that, in view of \eqref{K}, it is easily seen that \eqref{intstab} holds, with  different constants $C$ and $\eta$, also when the $L^1$ norm is replaced by any $L^p$ norm, with $p<\infty$. 

Also, it is worth noticing that, if for a given $d>0$ it is known in addition that $q_1=q_2$ in $\Omega\setminus \overline{\Omega_d}$, then 
\eqref{bdryerror} is automatically satisfied when \eqref{interror} holds true. Obviously, in such a case, \eqref{intstab} can be improved to
\begin{equation*}
 \|q_1-q_2\|_{L^{1}(\Omega)} \le C\left(\varepsilon^{1/2}+\varepsilon \right)^{\eta} \ .
\end{equation*}
\end{remark}

In the next Section 2 we shall state two theorems, Theorem \ref{weight} and Theorem \ref{negint}, which constitute the main tools for the proof of Theorem \ref{main}, which is also given there. Sections 3 and 4 contain the proofs of Theorem \ref{weight} and Theorem \ref{negint}, respectively. 

\section{Proof of the main theorem}
We begin with a weighted stability estimate on the electric field intensity $|u|$.
\begin{theorem}\label{weight}
Let the assumptions of Theorem \ref{main} be satisfied. We have
\begin{equation}\label{wstab}
\int_{\Omega}\left(|u_1|+|u_2|\right)\left(|u_1|-|u_2|\right)^2 \leq C \varepsilon 
\end{equation}
where $C>0$ depends only on $K, E$ and on $|\Omega|$.
\end{theorem}
The proof is postponed to the next Section 3.
\begin{remark}
 Note that also a $L^3$ stability estimate, without weight, follows easily from \eqref{wstab}
\begin{equation*}
 \int_{\Omega}\left||u_1|-|u_2|\right|^3 \leq C \varepsilon \ .
\end{equation*}
\end{remark}
The following theorem consists of a quantitative form of the strong unique continuation property for solutions to equation \eqref{sch}.
\begin{theorem}\label{negint} Let $\Omega$ be of {Lipschitz class} with
constants $\rho$, $M>0$.
Let $q\in L^{\infty}(\Omega)$ satisfy \eqref{K} and let $u\in W^{1,2}(\Omega)\cap C(\overline{\Omega})$ be a solution to \eqref{sch}. Assume that $u$ satisfies the a priori assumptions \eqref{E}, \eqref{H}. 
Then, for every $d>0$, there exists $\delta ,  C>0$, only depending on $d, K, E, H$ and on $\rho, M, |\Omega|$ such that
\begin{equation}\label{eq:negint}
 \int_{\Omega_d} |u|^{-\delta}\le C \ .
\end{equation}
\end{theorem}
The proof can be found in the final Section 4.

Assuming the above two theorems proven, we can now complete the proof of Theorem \ref{main}.
\begin{proof}[Proof of Theorem \ref{main}]
We compute
\begin{equation*}
 (q_1-q_2)u_1^2= (q_1u_1^2-q_2u_2^2) - q_2(u_1^2 - u_2^2)
\end{equation*}
hence
\begin{equation*}
\int_{\Omega} |q_1-q_2|u_1^2\leq |\Omega|\varepsilon+ \int_{\Omega}q_2\left(|u_1|+|u_2|\right)||u_1|-|u_2||
\end{equation*}
and, by H\"older's inequality,
\begin{equation*}
\int_{\Omega} |q_1-q_2|u_1^2\leq C\left[\varepsilon+ \left(\int_{\Omega}(|u_1|+|u_2|)\right)^{1/2}\left(\int_{\Omega}\left(|u_1|+|u_2|\right)\left(|u_1|-|u_2|\right)^2\right)^{1/2}\right]
\end{equation*}
where $C>0$ only depends on $|\Omega|$ and on $K$, now using \eqref{E} and Theorem \ref{weight} we obtain
\begin{equation}\label{weightq}
\int_{\Omega} |q_1-q_2|u_1^2\leq C\left(\varepsilon+\varepsilon^{1/2}\right)
\end{equation}
where $C>0$ is a new constant which only depends on $|\Omega|, K$ and on $E$. Now, fixing $d>0$ and choosing $\delta>0$ according to Theorem \ref{negint}, by H\"older's inequality we get
\begin{equation*}
\int_{\Omega_d}|q_1-q_2|^{\frac{\delta}{\delta+2}}  \le \left(\int_{\Omega_d}|u_1|^{-\delta}\right)^{\frac{2}{\delta+2}}  \left( \int_{\Omega} |q_1-q_2|u_1^2\right)^{\frac{\delta}{\delta+2}} 
\end{equation*}
finally, applying Theorem \ref{negint} to $u_1$ and using \eqref{K}, \eqref{weightq} we arrive at \eqref{intstab} with $\eta = \frac{\delta}{\delta+2}$.
\end{proof}
\begin{remark}
Inequality \eqref{weightq} might also be used to obtain a localized H\"older stability with a uniform exponent on regions where the measurement $qu^2$ is bounded away from $0$. For any $t>0$, set $D_t =\left\{x\in \Omega| q_1u_1^2\ge t\right\}$ then \eqref{weightq} implies
\begin{equation*}
\|q_1-q_2\|_{L^{1}(D_t)} \leq \frac{KC}{t}\left(\varepsilon+\varepsilon^{1/2}\right) \ .
\end{equation*}
\end{remark}
\section{The weighted estimate on the electric field}
\begin{proof}[Proof of Theorem \ref{weight}]
Denote by $N_i=\left\{x\in \Omega| u_i(x) =0\right\}$ the nodal set of $u_i$, $i=1,2$. Let us remark that, by the continuity of $u_i$, $N_i$ is a closed set, furthermore, by the unique continuation property for \eqref{sch}, we also know that $N_i$ has zero Lebesgue measure. We decompose $\Omega \setminus (N_1\cup N_2)$ into its connected components $\Omega_j$, we recall that such components are open and countably many. We observe that, by \eqref{interror} and by the continuity of $u_1,u_2$
\begin{equation*}
u_2^2\leq {K\varepsilon} + K^2u_1^2 \, \text{  everywhere in } \overline{\Omega} \ , 
\end{equation*}
consequently, we have $u_2^2\leq K\varepsilon$ on $N_1$, and analogously, $u_1^2\leq K\varepsilon$ on $N_2$. Therefore
\begin{equation*}
||u_1|-|u_2||\leq \sqrt{K\varepsilon} \, \text{ on } N_1\cup N_2 \ , 
\end{equation*}
hence, noticing that $\partial \Omega_j \subset N_1\cup N_2 \cup \partial \Omega$ and using \eqref{bdryerror} we have
\begin{equation}\label{bdryj}
||u_1|-|u_2||\leq \sqrt{K\varepsilon} \, \text{ on } \partial \Omega_j \text{ for every } j \ . 
\end{equation}
Let us fix one component $\Omega_j$. Note that in $\Omega_j$ $u_1, u_2$ have constant sign, possibly different. Since our aim is to estimate the difference $|u_1|-|u_2|$, we are allowed to change the signs of $u_1, u_2$ and may assume, without loss of generality, that $u_1, u_2$ are both positive in $\Omega_j$. We introduce the function
\begin{equation*}
\varphi^+=\left[u_1- u_2 - 2 \sqrt{K\varepsilon}\right]^+ \ , 
\end{equation*}
where $\left[\cdot\right]^+$ denotes the positive part. By \eqref{bdryj} and by continuity, we have $|u_1- u_2|<2 \sqrt{K\varepsilon}$ on a neighborhood of  $\partial \Omega_j$, hence $\varphi^+=0$ near $\partial \Omega_j$. If we define $\psi_i= u_i\varphi^+$, $i=1,2$, we obtain that $\psi_i \in W_0^{1,2}(\Omega_j)$, hence by the weak formulation of \eqref{sch} we obtain
\begin{eqnarray*}
\int_{\Omega_j} \nabla u_1 \cdot \nabla \psi_1 &=&\int_{\Omega_j} q_1u_1 \psi_1 =\int_{\Omega_j} q_1u_1^2 \varphi^+  \ , \\
\int_{\Omega_j} \nabla u_2 \cdot \nabla \psi_2 &=&\int_{\Omega_j} q_2u_2^2 \varphi^+ \ ,
\end{eqnarray*}
and subtracting
\begin{eqnarray*}
\int_{\Omega_j} \left[\nabla u_1 \cdot \nabla (u_1 \varphi^+) -\nabla u_1 \cdot \nabla (u_2 \varphi^+) +\nabla u_1 \cdot \nabla (u_2 \varphi^+)- \nabla u_2 \cdot \nabla (u_2 \varphi^+)\right] = \\ = \int_{\Omega_j} \left(q_1u_1^2-q_2u_2^2\right) \varphi^+ \ .
\end{eqnarray*}
Consequently, denoting
\begin{eqnarray*}
I_1&=&\int_{\Omega_j} \nabla u_1 \cdot \nabla \left((u_1-u_2) \varphi^+\right)\ ,\\I_2&=&\int_{\Omega_j}\nabla (u_1-u_2) \cdot \nabla (u_2 \varphi^+)\ ,  \\ I_3&=& \int_{\Omega_j} \left(q_1u_1^2-q_2u_2^2\right) \varphi^+ \ ,
\end{eqnarray*}
we have
\begin{equation*}
I_1+I_2 = I_3 \ .
\end{equation*}
Recalling once more the weak formulation of \eqref{sch}, we evaluate
\begin{eqnarray*}
I_1=\int_{\Omega_j} q_1 u_1 (u_1-u_2) \varphi^+ = \\ = \int_{(u_1-u_2)>2\sqrt{K\varepsilon}} q_1 u_1 (u_1-u_2) (u_1-u_2-2\sqrt{K\varepsilon})\ge\\ \ge \frac{1}{K}\int_{(u_1-u_2)>2\sqrt{K\varepsilon}}|u_1|  (u_1-u_2-2\sqrt{K\varepsilon})^2 \ .
\end{eqnarray*}
Note that here, and in what follows, it is understood that the domain of integration is  a subset of $\Omega_j$.
Next, we compute
\begin{eqnarray*}
I_2=\int_{(u_1-u_2)>2\sqrt{K\varepsilon}} u_2 |\nabla(u_1-u_2)|^2 +\\+ \int_{(u_1-u_2)>2\sqrt{K\varepsilon}}(u_1-u_2-2\sqrt{K\varepsilon})\nabla(u_1-u_2-2\sqrt{K\varepsilon})\cdot \nabla u_2
\ge \\\ge\frac{1}{2}\int_{(u_1-u_2)>2\sqrt{K\varepsilon}}\nabla(u_1-u_2-2\sqrt{K\varepsilon})^2 \cdot \nabla u_2=\\= \frac{1}{2}\int_{(u_1-u_2)>2\sqrt{K\varepsilon}}q_2 u_2  (u_1-u_2-2\sqrt{K\varepsilon})^2 \ge\\ \ge \frac{1}{2K}\int_{(u_1-u_2)>2\sqrt{K\varepsilon}}|u_2|  (u_1-u_2-2\sqrt{K\varepsilon})^2\ ,
\end{eqnarray*}
hence, adding up
\begin{eqnarray*}
I_1+I_2\ge  \frac{1}{2K}\int_{(u_1-u_2)>2\sqrt{K\varepsilon}}\left(|u_1|+|u_2|\right)  (u_1-u_2-2\sqrt{K\varepsilon})^2\ .
\end{eqnarray*}
Regarding the third integral, we observe that
\begin{eqnarray*}
I_3 \le \varepsilon\int_{(u_1-u_2)>2\sqrt{K\varepsilon}} \left(|u_1|+|u_2|+2\sqrt{K\varepsilon}\right) \le\\ \le 2\varepsilon\int_{(u_1-u_2)>2\sqrt{K\varepsilon}} \left(|u_1|+|u_2|\right) \ .
\end{eqnarray*}
Consequently, we obtain
\begin{eqnarray*}
\int_{(u_1-u_2)>2\sqrt{K\varepsilon}}\left(|u_1|+|u_2|\right)  (u_1-u_2-2\sqrt{K\varepsilon})^2 \le\\ \le 4K\varepsilon\int_{(u_1-u_2)>2\sqrt{K\varepsilon}} \left(|u_1|+|u_2|\right) \ ,
\end{eqnarray*}
and the analogous estimate holds if we interchange the roles f $u_1$ and $u_2$. Hence we arrive at
\begin{eqnarray*}
\int_{||u_1|-|u_2||>2\sqrt{K\varepsilon}}\left(|u_1|+|u_2|\right)  (||u_1|-|u_2||-2\sqrt{K\varepsilon})^2 \le\\ \le 4K\varepsilon\int_{||u_1|-|u_2||>2\sqrt{K\varepsilon}} \left(|u_1|+|u_2|\right) \ ,
\end{eqnarray*}
now, by triangle inequality
\begin{eqnarray*}
(|u_1|-|u_2|)^2 \le \left(\left|||u_1|-|u_2||-2\sqrt{K\varepsilon}\right|+2\sqrt{K\varepsilon}\right)^2 \le\\ \le 2\left(\left|||u_1|-|u_2||-2\sqrt{K\varepsilon}\right|^2+4K\varepsilon\right) \ ,
\end{eqnarray*}
therefore
\begin{eqnarray*}
\int_{||u_1|-|u_2||>2\sqrt{K\varepsilon}}\left(|u_1|+|u_2|\right)  \left(|u_1|-|u_2|\right)^2 \le\\ \le 16K\varepsilon\int_{||u_1|-|u_2||>2\sqrt{K\varepsilon}} \left(|u_1|+|u_2|\right) \ .
\end{eqnarray*}
It remains to consider the subset of $\Omega_j$ where $||u_1|-|u_2||\le 2\sqrt{K\varepsilon}$, in this case it is a straightforward matter to obtain
\begin{eqnarray*}
\int_{||u_1|-|u_2||\le2\sqrt{K\varepsilon}}\left(|u_1|+|u_2|\right)  \left(|u_1|-|u_2|\right)^2 \le\\ \le 4K\varepsilon\int_{||u_1|-|u_2||\le2\sqrt{K\varepsilon}} \left(|u_1|+|u_2|\right) \ ,
\end{eqnarray*}
consequently
\begin{eqnarray*}
\int_{\Omega_j}\left(|u_1|+|u_2|\right)  \left(|u_1|-|u_2|\right)^2 \le\\ \le 16K\varepsilon\int_{\Omega_j} \left(|u_1|+|u_2|\right) \ ,
\end{eqnarray*}
and adding up with respect to $j$
\begin{eqnarray*}
\int_{\Omega}\left(|u_1|+|u_2|\right)  \left(|u_1|-|u_2|\right)^2 \le\\ \le 16K\varepsilon\int_{\Omega} \left(|u_1|+|u_2|\right) \ ,
\end{eqnarray*}
finally using \eqref{E} we arrive at \eqref{wstab}.
\end{proof}

\section{Quantitative estimates of unique continuation}
We begin by recalling the following version of a quantitative estimate of unique continuation, which is well-known to be very useful in the treatment of various inverse problems, see for instance \cite{alrossSIAM},  a proof can be found in \cite[Theorem 5.3]{arrv}.
\begin{theorem} [Lipschitz propagation of smallness]
  \label{theo:LipPropSmall}
  Let the assumptions of Theorem \ref{negint} be satisfied. For every $r >0$ and for
  every $x \in \Omega_{r}$, we have
\begin{equation}
  \label{eq:p2-e4}
  \int_{B_{r}(x)} u^2  \geq C
  \int_{\Omega}  u^2  \ ,
\end{equation}
where  $C>0$ only depends on $r, K, E/H$ and on 
 $\rho, M, |\Omega|$ .
\end{theorem}

Also the following theorem is a manifestation of the strong unique continuation property, its original version is due to Garofalo and Lin \cite{garofalolin1, garofalolin2} , the present global formulation is indeed a consequence of the previous Theorem \ref{theo:LipPropSmall}, for the details of a proof we may refer to \cite[Theorem 3.4]{amrv}.
 \begin{theorem} [Doubling inequality] \label{doub}
  Let the assumptions of Theorem \ref{negint} be satisfied. For every $\overline{r} >0$ and for
  every $x \in \Omega_{2\overline{r} }$, we have
\begin{equation}
  \label{doubin}
  \int_{B_{2r}(x)}  u^2   \le C
  \int_{B_{r}(x)}  u^2   \text{ for every } r\le \overline{r} \ ,
\end{equation}
where  $C>0$ only depends on $\overline{r} , K, E/H$ and on 
 $\rho, M, |\Omega|$ .
\end{theorem}

\begin{proof}[Proof of Theorem \ref{negint}]
Since Garofalo and Lin \cite{garofalolin1}, it is well-known that, as a consequence of the above stated doubling inequality and of the standard local boundedness estimates for solutions to \eqref{sch} \cite[Theorem 8.17]{gilbargtrud}, $u^2$ turns out to be a Mukenhoupt weight, Coifman and Fefferman \cite{coifeff}. More specifically, we obtain that for every $\overline{r} >0$ and for
  every $x \in \Omega_{2\overline{r} }$, there exists $p>1 , C>0$, only depending on $\overline{r}, K, E/H$ and on 
 $\rho, M, |\Omega|$  such that for every  $x \in \Omega_{2\overline{r} }$ and for every $r\le \overline{r}$ we have
	\begin{equation}\label{A_p}
\left(\frac{1}{|B_r(x)|}\int_{B_{r}(x)} u^2\right) \left(\frac{1}{|B_r(x)|}\int_{B_{r}(x)} |u|^{-\frac{2}{p-1}}\right)^{p-1}\le C \ .
\end{equation}
It is a straightforward matter to construct a covering of $\Omega_d$ with balls $B_{d/4}(x_i)$, $i=1,\ldots, N$ such that their doubles $B_{d/2}(x_i)$ stay within $\Omega$ and their number $N$ is is dominated, up to an absolute constant, by $|\Omega|d^{-n}$. Using \eqref{A_p} for each $B_{d/4}(x_i)$ we get
\begin{equation*}
\int_{B_{\frac{d}{4}}(x_i)} |u|^{-\frac{2}{p-1}}\le C |B_{\frac{d}{4}}(x_i)|^{1-\frac{1}{p-1}}\left(\int_{B_{\frac{d}{4}}(x_i)} u^2\right)^{-\frac{1}{p-1}}\ ,
\end{equation*}
hence, recalling \eqref{eq:p2-e4} and adding up with respect to $i=1,\ldots, N$, we arrive at \eqref{eq:negint} with $\delta =\frac{2}{p-1}$.
\end{proof}

\end{document}